\documentclass[12pt]{amsart}
\usepackage{a4wide}
\usepackage{amsmath,amssymb,amsfonts}
\usepackage{float,epsfig,graphicx,graphics,mathrsfs,dsfont}
\usepackage{color,cite}
\newtheorem{definition}{Definition}[subsection]
\newtheorem{note}{Note}[section]
\newtheorem{proposition}{Proposition}[section]
\newtheorem{theorem}{Theorem}[section]
\newtheorem{lemma}{Lemma}[section]
\newtheorem{example}{Example}[section]
\title{Non-affine fractal hypersurfaces: construction and dimensions}
\author[A. Hossain]{ A. Hossain\textsuperscript{1}}

\address{$^1$Department of Mathematics, Presidency University, 86/1, College Street, Kolkata, 700 073, West Bengal, India.}
\curraddr{}
\email{hossain4791@gmail.com}

\author[J. Buescu ]{ J. Buescu\textsuperscript{2}}

\address{$^2$Faculdade de Ciências da Universidade de Lisboa and CMAFcIO-Centro de Matemática, Aplicações Fundamentais e Investigação Operacional, Campo Grande, 16, Lisboa 1749-016, Portugal.}
\curraddr{}
\email{jsbuescu@fc.ul.pt}

\date{November, 2024}

\begin{document}
	\maketitle
	\begin{abstract}
		This article presents the construction of a non-affine hypersurface on an $n$-simplex in $\mathbb{R}^n$. Additionally, fractal dimension of the graph of a non-affine multivariate real-valued fractal function is estimated under certain conditions. Furthermore, the upper bound of the Hausdorff dimension of the invariant probability measure supported on the graph of such fractal function is estimated. 
			\end{abstract}
	\textbf{Keywords:} Iterated function system, Fractal hypersurface, Regular $n$-simplex, Non-affine fractal hypersurface, Fractal dimension, Invariant probability measure.\\
	\textbf{Mathematics Subject Classification:} Primary; 28A80, Secondary; 41A30, 37A50 
	\section{Introduction}
	Fractal geometry is a field of study that primarily focuses on understanding and describing the complex patterns and structures found in natural phenomena and objects, such as clouds, trees, mountains, coastlines, and even the cells of the human body \cite{Mandelbrot1982,Falconer2004,Coleman1992}. The world of fractals is home to many iconic and well-known examples, including the Koch curve, Cantor set, Cantor dust, Sierpiński gasket, and many more. These fractals have become synonymous with the field of fractal geometry and continue to inspire research and fascination. One key tool in fractal geometry is the theory of iterated function systems (IFSs), which offers powerful methods for generating and modeling fractals, allowing researchers to simulate and analyze these intricate patterns \cite{Barnsley2014,Hutchinson1981,Serpa2017}. The attractors generated by IFSs are typically fractal sets, characterized by their unique and intricate geometric structures \cite{Barnsley2014,Falconer2004,Buescu2012,Akhtar2022,DAniello2016}. Building on this concept, Barnsley \cite{Barnsley1986} introduced in 1986 the idea of fractal interpolation functions (FIFs), which are specifically generated by IFSs. This innovation enabled the creation of functions that can accurately model and reproduce the complex patterns found in fractal geometry \cite{Dalla2002,Massopust1990,chand2015approximation,Chand2015,hossain2023fractal}. Those functions are used to compress images by exploiting the self-similarity properties of fractals \cite{Fisher1994}, utilized in computer graphics to generate and simulate natural landscapes, such as mountains, rivers, and clouds, which have inherent fractal characteristics \cite{Barnsley2014,Barnsley1993,Solomyak2024}. Navascu{\'e}s \cite{Navascues2005} introduced the concept of non-affine fractal functions, expanding the field of fractal geometry and opening up new avenues for research and applications. Non-affine fractal functions do not exhibit affine self-similarity, meaning their scaling properties are not uniform in all directions. These functions have been used to model and analyze complex phenomena in various fields, used to approximate trigonometric polynomials, to analyze and process signals with complex, non-stationary behavior, to generate realistic models of natural objects and environments \cite{Peters1994}, or to model the growth and branching patterns of biological systems, such as blood vessels and trees. Many authors generalized these concepts by defining various fractal interpolation surfaces (FISs) on different types of regions, e.g. FISs on rectangular grids \cite{Dalla2002}, FIS on a triangular region without edge condition \cite{Massopust1990}, recurrent FISs on rectangular grids \cite{Liang2021}, non-affine FIS on a rectangle \cite{Navascues2020}, etc.
	In \cite{massopust2024fractal}, Massopust pioneered the development of multivariate, real-valued affine fractal functions defined on a regular $n$-simplex in $\mathbb{R}^n$, constructing an affine fractal basis for these functions which enables the representation of complex fractal structures. The graphical representation of these fractal functions is termed as affine fractal hypersurfaces.\\ 
	Motivated by these results, in this article we introduce a new class of multivariate, real-valued non-affine fractal functions defined on an $n$-simplex in $\mathbb{R}^n$, and we term the graph of such a function a \textbf{non-affine fractal hypersurface}. In addition, we have included a graphical representation of this type of function, providing concrete examples that facilitate a deeper understanding of their characteristics and functionality (see Figures~\ref{fig3} and \ref{fig4}).\\
	  Fractal dimensions are mathematical concepts used to describe the complexity and scaling properties of fractals. Many authors have studied the fractal dimension of graphs of different fractal functions in the literature \cite{nussbaum2012positive,Peters1994,sahu2020box,Akhtar2016,Hossain2023a,Jiang2023}.
	   In \cite{Buescu2019}, Buescu et al. explored systems of non-affine iterative functional equations, deriving bounds for the Hausdorff dimension of the solution's graph. Additionally, they elegantly connected these findings to related concepts in the literature, including Girgensohn functions, fractal interpolation functions, and Weierstrass functions, revealing a rich web of relationships between these mathematical objects. Verma at al. \cite{Verma2023} constructed more general non-affine FIFs on the Sierpiński gasket by taking variable scaling factors. Liang at al. \cite{Liang2021} also provided bounds for box dimensions of the graph of recurrent FISs for equally-spaced data sets. In this article in a more general setting we estimate the bound of fractal dimension of a non-affine fractal hypersurface on a regular $n$-simplex. We also study the Hausdorff dimension of the invariant probability measure supported on the graph.
	 
 \section{Preliminaries}
 For clarity and ease of understanding, this section will cover essential definitions and notations. Additional information may be found in references \cite{Barnsley2014,massopust2024fractal,Hutchinson1981}.
\subsection{Iterated function system}
Let $(X,\lVert\cdot\rVert_{X})$ be a Banach space and $d$ be the metric induced by this norm. Consider
    $$\mathscr{H}(X)=\left\{ K\subset X:~K\neq \emptyset~\mbox{and}~K~\mbox{is compact}\right\}$$
    endowed with the Hausdorff metric $H_d$, defined by $$H_d(A, B)=\max \left\{ d(A,B),~d(B,A)\right \}$$ for all $A, B \in \mathscr{H}(X),$ where $d(A,B)=\sup_{x\in A} \inf_{y\in B}d(x, y)$. The space $\left(\mathscr{H}(X),~H_d\right)$ is complete if $(X,d)$ is complete \cite{Barnsley2014}. Let $W_n:X\rightarrow X$, for $n=1,2,\ldots,N$, be continuous functions; then $\mathscr{W}=\left\{\big(X;W_n\big):~n=1,2,\ldots,N\right\}$ is called an IFS \cite{Barnsley2014,Falconer2004}. If, for each $n=1,2,\ldots,N$, the $W_n$ are contractive maps, that is, if there exist $s_n\in[0,1($ such that
    \begin{equation*}
        d\left(W_n(x),W_n(y)\right)\leq s_n d(x,y)
    \end{equation*}
for all $x,y\in X$, then the corresponding IFS $\mathscr{W}=\left\{\big(X;W_n\big):~n=1,2,\ldots,N\right\}$ is known as a hyperbolic IFS. In these conditions, the set-valued  Hutchinson operator  $W : \mathscr{H}(X) \rightarrow \mathscr{H}(X)$, given by $W(B)=\bigcup_{n=1}^N W_n(B)~ \mbox{for all}~ B \in \mathscr{H}(X)$, is also a contraction map with contractivity factor $s=\max\{s_n:~n=1,2,\ldots,N\}$. 
Define $W^0(B) = B$ and let $W^k(B)$ denote the $k$-fold composition of $W$ applied to $B$.
\begin{definition}(See \cite{barnsley2013developments})
\label{dfnbasin}
	A compact subset $F$ of $(X,d)$ is called an \textbf{attractor} of an IFS $\mathscr{W}=\left\{\big(X;W_n\big):~n=1,2,\ldots,N \right\}$ if
	\begin{enumerate}
		\item $W(F)=F$ and
		\item \label{itemboa} there exists an  open subset $U$ of $X$ such that $F\subset U$ and
		$$\lim\limits_{k\rightarrow \infty}W^k(B)=F \quad \mbox{for all}~ B\in \mathscr{H}(U),$$
		where convergence is with respect to the Hausdorff metric $H_d$ on $\mathscr{H}(X).$	
	\end{enumerate}
\end{definition}
\begin{note}
	The largest open set $U$ in Definition~\ref{dfnbasin} is known as the \textbf{basin} of attraction for the attractor $F$ of the IFS $\mathscr{W}$ and is denoted by $B(F)$. 
\end{note}
\begin{definition}[Hausdorff dimension]
	Let $\big(X,~d\big)$ be metric space. Then the Hausdorff dimension of a set $F\subset X$ is given by 
	\begin{equation*}
		\dim_HF=\inf\left\{s>0:~\forall\delta>0,~\mbox{there is countable cover}~\left\{U_i\right\}_{i\in\mathbb{N}}~\mbox{of}~F~\mbox{such that}~\sum_{i\in\mathbb{N}}\lvert U_i\rvert^s<\delta\right\},
	\end{equation*}
	where $\lvert U_i\rvert$ denotes the diameter of $U_i$.
\end{definition}
\begin{definition}[Box-counting dimension]
	Let $(X,d)$ be a metric space and $F$ be a compact subset of $X$. Let $N_r(F,d)$ be the minimum number of balls of radius $r$ that cover $F$. Then the upper and lower box dimensions of $F$ are defined by
	\begin{align*}
		\overline{\dim}_B(F,d)=&\limsup\limits_{r\rightarrow 0}\frac{\log N_r(F,d)}{-\log r}\\ \underline{\dim}_B(F,d)=&\liminf\limits_{r\rightarrow 0}\frac{\log N_r(F,d)}{-\log r}.
	\end{align*}
	If both exist and are equal, the common value is called the box dimension of $F$ and is denoted by $\dim_B(F,d)$ \cite{Falconer2004}.
\end{definition}

\begin{definition}
	If $\mu$ is a Borel probability measure on $X$, then the Hausdorff dimension of $\mu$ is given by 
	\begin{equation}\label{mmdfg}
		\dim_H(\mu)=\inf\big\{\dim_HF:~F~\mbox{is a Borel subset such that}~\mu(F)>0\big\}.
	\end{equation}
\end{definition}
\subsection{Fractal Surfaces}
Let $(X,\lVert\cdot\rVert_{X})$ and $(Y,\lVert\cdot\rVert_{Y})$ be two Banach spaces. In this section, a class of special attractors of IFSs, namely attractors that are the graphs of bounded functions $f:\Delta\subset X\to Y$, where $\Delta\in\mathscr{H}(X)$, is provided (see \cite{massopust2024fractal}). Suppose there exists a collection of injective maps $\{L_i:\Delta\to \Delta,~i=1,2,\ldots,N\}$ such that 
\begin{align*}
   &\{L_i(\Delta):~i=1,2,\ldots,N\}~\mbox{ is a set-theoretic partition of $\Delta$, that is}\\ 
   & \Delta=\bigcup_i^NL_i(\Delta)~ \mbox{and}~\big( L_i(\Delta)\big)^\circ\cap \big(L_j(\Delta)\big)^\circ=\emptyset,~\mbox{for all} ~i\neq j,
\end{align*}
where $(A)^\circ$ denotes the interior of the set $A$. 
Let $\mathcal{B}(\Delta)=\left\{f:\Delta\to Y:~f~\mbox{is bounded}\right\}$ and for all $f\in\mathcal{B}(\Delta)$, define a norm $\lVert f\rVert_{\infty,\Delta}:=\sup_{x\in\Delta}\lVert f(x)\lVert_Y$. 
It is straightforward to show that $\big(\mathcal{B}(\Delta),\lVert f\rVert_{\infty,\Delta}\big)$ is a Banach space.
For $i=1,2,\ldots,N$, let $F_i:\Delta\times Y\to Y$ be a mapping which is contractive with respect to the second variable, i.e., there exists $0\leq c<1$ such that 
\begin{equation*}
    \lVert F_i(x,y_1)-F_i(x,y_2)\rVert_Y\leq c\lVert y_1-y_2\rVert_Y, \quad \forall x\in \Delta~\mbox{and}~\forall y_1,y_2\in Y.
\end{equation*}
Define a Read-Bajactarevi\'{c} (RB)-operator $\mathcal{T}:\mathcal{B}(\Delta)\to \mathcal{B}(\Delta)$ by
\begin{equation}\label{1rbopt}
   \mathcal{T}(f)=\sum_{i=1}^N F_i\big(L_i^{-1}(x),f\circ L_i^{-1}(x)\big)\cdot\chi_{\Delta_i}, 
\end{equation}
where $\chi_A$ denotes the characteristic function of $A$, which takes the value one on $A$ and zero outside  $A$. Then $\mathcal{T}$ is well defined and is also a contraction map on the Banach space $\mathcal{B}(\Delta)$, thus having a unique fixed point $f$ in  $\mathcal{B}(\Delta)$. This unique fixed point is called the multivariate fractal function and its graph is a fractal surface on $X\times Y$; sometimes it is called \textbf{fractal hypersurface} on $X\times Y$. The graph of this function $f$ is an attractor of the IFS $\left\{\big(\Delta\times Y;W_i\big):~i=1,2,\ldots,N\right\}$, where the $W_i$'s are given by 
\begin{equation*}
    W_i(x,y)=\big(L_i(x),F_i(x,y)\big).
\end{equation*}
For more details see Massopust \cite{massopust2024fractal}. 

\subsection{Affinely generated fractal surfaces in $\mathbb{R}^{n+1}$}
	In this section, we deal mainly with the connections between IFS and multivariate real-valued affine FIF. 
 \begin{definition}
    Let $\{e_0,e_1,\ldots,e_n\}$ be a set of affinely independent points in $\mathbb{R}^n$. A regular n-simplex on $\mathbb{R}^n$ is defined as the point set
	\begin{equation*}
		\Delta:=\left\{x\in \mathbb{R}^n:x=\sum_{k=0}^nt_ke_k;0\leq t_k\leq 1;\sum_{k=0}^nt_k=1\right\}.
	\end{equation*}.
 \end{definition}
 Over the $n$-simplex $\Delta$, consider 
	\begin{align*}
		\mathcal{C}(\Delta)=\left\{f:\Delta\to\mathbb{R}~\mbox{such that $f$ is continuous on $\Delta$}\right\}.
	\end{align*}
Then the space $\left(\mathcal{C}(\Delta),d_{\infty,\Delta}\right)$ forms a complete metric space, where the metric $d_{\infty,\Delta}$ is induced by the sup norm, defined as $\lVert f\rVert_{\infty,\Delta}:=\sup_{x\in\Delta}\lvert f(x)\lvert$ for $f\in\mathcal{C}(\Delta)$.\\

Now, let $\{\Delta_i:i=1,2,\ldots,N\}$ be a collection of non-empty compact subsets of $\Delta$ with the properties:
\begin{align*}
    &(A1)\quad \Delta=\bigcup_{i}^N\Delta_i;\\
    &(A2)\quad \Delta_i ~\mbox{is similar to}~ \Delta,~i=1,2,\ldots,N;\\
    &(A3)\quad \Delta_i~\mbox{is congruent to}~\Delta_j~\mbox{with}~(\Delta_i)^\circ\cap (\Delta_j)^\circ=\emptyset ~\mbox{for all}~i,j\in\{1,2,\ldots,N\}.
\end{align*}
Then there exist $N$ contractive similarity maps $L_i:\Delta\to \Delta_i$ given by 
	\begin{equation}\label{limaps}
		L_i=c_iO_i+t_i,
	\end{equation}
	where $c_i<1$ is the similarity constant or the similarity ratio for $\Delta_i$ with respect to
	$\Delta$, $O_i$ is an orthogonal transformation on $\mathbb{R}^n$, and $t_i$ is  a translation in $\mathbb{R}^n$.\\

 Let $V$ be the set of vertices of $\Delta$ and $V_i$ be the set of vertices of $\Delta_i$. Let $l:\bigcup V_i\to V$ be a {\bf labeling map}, defined in such a way that the condition
	\begin{equation*}\label{labving}
		L_i(l(v))=v
	\end{equation*}
 is satisfied for all $i=1,2,\ldots,N$ and for all $v\in\bigcup V_i$. Consider the interpolation set 
	$$\{(v,z_v)\in\Delta\times \mathbb{R}:v \in \bigcup V_i\}.$$
	For $i=1,2,\ldots,N$, define the continuous maps $F_i:\Delta\times\mathbb{R}\to\mathbb{R}$
	\begin{equation}\label{fnvg}
		F_i(x,y)=\lambda_i(x)+\alpha_iy,
	\end{equation}
	 where $\lambda_i:\Delta\to \mathbb{R}$ are affine maps and $\alpha_i\in(-1,1)$. The affine map $\lambda_i$ is uniquely determined by the interpolation conditions
	 \begin{equation*}
	 	\lambda_i(l(v))+\alpha_i z_{l(v)}=z_v.
	 \end{equation*}
	 Impose the following join-up conditions:
	 \begin{align*}
	 	\lambda_i\circ L_i^{-1}(s,t)+\alpha_i f\circ L_i^{-1}(s,t)=\lambda_j\circ L_j^{-1}(s,t)+\alpha_j f\circ L_j^{-1}(s,t)
	 \end{align*}
 for all $(s,t)\in E_{ij}:=\Delta_i\cap\Delta_j$, $i\neq j$ and for all continuous functions $f:\Delta\to \mathbb{R}$ (see \cite{massopust2024fractal,Serpa2015,Buescu2021,Bedford2014}). The set $E_{ij}$ is called a common edge of $\Delta_i$ and $\Delta_j$.\\
 
	  Finally, the IFS is generated by the mappings $L_i$ and $F_i$, having an attractor which is a graph of the continuous map $f:\Delta\to \mathbb{R}$. The map $f$ is called a {\bf multivariate
		real-valued affine fractal function} and its graph is known as {\bf affinely generated fractal
		hypersurface} or an {\bf affine fractal hypersurface} \cite{massopust2024fractal}.
 
	\section{ Construction of non-affine fractal hypersurface.}\label{conafhsfr}
 In this section, we present the construction of a non-affine multivariate fractal function on an $n$-simplex. \\
 
 Let $\Delta$ be an $n$-simplex in $\mathbb{R}^n$ and $\left\{\Delta_i:~i=1,2,\ldots,N\right\}$ be the set-theoretic partition of $\Delta$ satisfying the conditions $(A1)$, $(A2)$ and $(A3)$, and $V_0$ be the vertex set of $\Delta$. Let $\alpha=\left(\alpha_1,\alpha_2,\ldots,\alpha_N\right)\in\mathbb{R}^N$ be a vector such that $\lvert\alpha_i\rvert <1$, which will act as a scaling vector. Let   $L_i:\Delta\to \Delta_i$ be the contractive similarity maps given in (\ref{limaps}). For $k\in\mathbb{N}$ and ${\bf i}=(i_1,i_2,\ldots,i_k)\in\{1,2,\ldots,N\}^k$, let $\Delta_{{\bf i}}:=L_{{\bf i}}(\Delta)=L_{i_1}\circ L_{i_2}\circ\cdots\circ L_{i_k}(\Delta)$ and $V_{\bf i}$ be the corresponding vertex set of $\Delta_{{\bf i}}$. Let us define $ Z_k:=\bigcup_{{\bf i}\in\{1,2,\ldots,N\}^k}V_{\bf i}$ and $Z_0=V_0$.\\

	 For a fixed $g\in \mathcal{C}(\Delta)$, let $g(v_k)=z_{v_k}$ for all $v_k\in Z_k$. Then the space
	\begin{align*}
	\mathcal{C}_{0}(\Delta)=\left\{f\in\mathcal{C}(\Delta): f(v_k)=g(v_k)=z_{v_k},~v_k\in Z_k\right\}
		\end{align*}
	is a closed subset of $\left(\mathcal{C}(\Delta),\lVert\cdot\rVert_{\infty,\Delta}\right)$, being therefore complete. Consider the interpolation set
	\begin{equation}
	\mathscr{Z}:=\big\{(v_k,z_{v_k})\in\Delta\times \mathbb{R}:v_k \in Z_k\big\}.
	\end{equation}
	Let us define a labeling map $l_k:Z_k \to Z_{k-1}$ for $k\in\mathbb{N}$ that satisfies the condition
 \begin{equation}\label{vstcnt}
		L_{\bf i}(l_k(v_k))=v_{k}
	\end{equation}
 for all ${\bf i}\in\{1,2,\ldots,N\}^k$, and for all $v_k\in Z_k$. For ${\bf i}\in\{1,2,\ldots,N\}^k$, consider the contraction homeomorphism $L_{\bf i}:\Delta\to\Delta_{{\bf i}}$ given by
 \begin{equation}\label{baseifs}
 L_{{\bf i}}(x)=L_{i_1}\circ L_{i_2}\circ\cdots\circ L_{i_k}(x)
 \end{equation}
 and the continuous maps $F_{\bf i}:\Delta\times\mathbb{R}\to\mathbb{R}$,
	\begin{equation}\label{frabu}
			F_{\bf i}(x,y)=\lambda_{\bf i}(x)+\alpha_{\bf i}y,
	\end{equation}
	where $\alpha_{\bf i}=\alpha_{i_1}\cdot\alpha_{i_2}\cdots\alpha_{i_k}$, $i_j\in\{1,\ldots,N\}$ for $j=1,\ldots,k$, and $\lambda_{\bf i}:\Delta\to \mathbb{R}$ are defined by 
	\begin{equation}\label{lasmahsy}
		\lambda_{\bf i}(x)=g\circ L_{\bf i}(x)-\alpha_{\bf i} b(x),
	\end{equation}
where the function $b:\Delta\to\mathbb{R}$ satisfies the condition $b(v_k)=g(v_k)$ for all $v_k \in Z_k$ and $b\neq g$. For all ${\bf i},{\bf j}\in\{1,2,\ldots,N\}^k$ and for all $x\in E_{\bf ij}=\Delta_{\bf i}\cap\Delta_{\bf j}$, we impose the following join-up condition
	\begin{equation}\label{edgrtg}
		\alpha_{\bf i}(f-b)\circ L_{\bf i}^{-1}(x)=\alpha_{\bf j}(f-b)\circ L_{\bf j}^{-1}(x)
	\end{equation} 
for all $f\in \mathcal{C}_{0}(\Delta)$. Since $\lvert\alpha_i\rvert<1$, for all $i\in\{1,2,\ldots,N\}$, therefore, $\lvert\alpha_{\bf i}\rvert<1$ for all ${\bf i}\in\{1,2,\ldots,N\}^k$, hence it follows that the map $F_{\bf i}$ is contractive with respect to the second variable and for all $v_k\in Z_k$,
\begin{equation}\label{coniuyt}
	F_{\bf i}\big(l_k(v_k),z_{l_k(v_k)}\big)=\lambda_{\bf i}(l_k(v_k))+\alpha_{\bf i} z_{l_k(v_k)}.
\end{equation}
Now,
\begin{align}\label{lmahsk}
	\lambda_{\bf i}(l_k(v_k))&=g\big(L_{\bf i}(l_k(v_k))\big)-\alpha_{\bf i} b(l_k(v_k))\\
\nonumber	&=g(v_k)-\alpha_{\bf i} g(l_k(v_k))\\
\nonumber	&=z_{v_k}-\alpha_{\bf i} z_{l_k(v_k)}.
\end{align}
From (\ref{coniuyt}) and (\ref{lmahsk}), we get
\begin{equation}
	F_{\bf i}\big(l_k(v_k),z_{l_k(v_k)}\big)=z_{v_k}.
\end{equation}
	Recalling the RB-operator $\mathcal{T}:\mathcal{C}_{0}(\Delta)\to \mathcal{C}_{0}(\Delta)$ defined in equation (\ref{1rbopt}) such that
	\begin{equation}\label{rboptr}
		\mathcal{T}(f)=\sum_{{\bf i}\in\{1,2,\ldots,N\}^k} g\cdot\chi_{\Delta_{\bf i}}+\sum_{{\bf i}\in\{1,2,\ldots,N\}^k}\alpha_{\bf i}\big(f-b\big)\circ L_{\bf i}^{-1}\cdot\chi_{\Delta_{\bf i}},
	\end{equation}
we then obtain the following result.
\begin{theorem}
	The operator $\mathcal{T}$ is well defined and contractive on $ \mathcal{C}_{0}(\Delta)$.
\end{theorem}
\begin{proof} It is easy to see that for ${\bf i}\in\{1,2,\ldots,N\}^k$, on each partition $\Delta_{{\bf i}}$, $\mathcal{T}(f)$ is continuous for all $f\in\mathcal{C}_{0}(\Delta)$. Let $v_k\in Z_k$.  Then $v_k\in V_{{\bf i}}\subset \Delta_{{\bf i}}$ for some ${\bf i}\in\{1,2,\ldots,N\}^k$. Therefore, from (\ref{vstcnt}), we get
	\begin{align*}
	(\mathcal{T}f)(v_k)&=g(v_k)+\alpha_{\bf i}\big(f-b\big)(L_{\bf i}^{-1}(v_k))\\
	&=g(v_k)+	\alpha_{\bf i}\big(f(l_k(v_k))-b(l(v_k))\big)=z_{v_k}.
	\end{align*}
Also for  $x\in E_{\bf ij}=\Delta_{\bf i}\cap\Delta_{\bf j}$, from (\ref{edgrtg}), we obtain
\begin{align*}
		(\mathcal{T}f)(x)=g(x).
\end{align*}
This shows that $\mathcal{T}$ is well defined and $(\mathcal{T}f)(v_k)=z_{v_k}$ for all $v_k \in Z_k$. Also, from (\ref{rboptr}), for all $f_1,f_2\in \mathcal{C}_{0}(\Delta)$ and all $x\in \Delta_{{\bf i}}$ we get
\begin{align*}
	\lvert \mathcal{T}(f_1)(x)-\mathcal{T}(f_2)(x)\rvert&=\lvert\alpha_{\bf i}\rvert \lvert f_1(x)-f_2(x)\rvert\\
		&\leq\alpha^k_{\infty}\lVert f_1-f_2\rVert_{\infty,\Delta},
\end{align*}
where, $\alpha_{\infty}=\max_{{1\leq i\leq N}}\{\vert \alpha_{i}\vert\}<1$. The above inequality is true for all $x\in\Delta=\bigcup_{{\bf i}\in\{1,2,\ldots,N\}^k}\Delta_{{\bf i}}$, hence taking the supremum over all $x\in\Delta$, we get 
	\begin{align}\label{rbcntv}
	\lVert \mathcal{T}(f_1)-\mathcal{T}(f_2)\rVert_{\infty,\Delta}\leq\alpha^k_{\infty}\lVert f_1-f_2\rVert_{\infty,\Delta}.
\end{align}
Since $\alpha^k_{\infty} <1$, $\mathcal{T}$ is contractive on $\mathcal{C}_{0}(\Delta)$. 
\end{proof}
The Banach fixed point theorem thus ensures that $\mathcal{T}$ has a unique fixed point $f^{\alpha}$ in $\mathcal{C}_{0}(\Delta)$. The function $f^{\alpha}$ is said to be a  {\bf multivariate real-valued non-affine fractal function} and its graph a {\bf non-affine fractal hypersurface}. Now, from (\ref{rboptr}), the function $f^{\alpha}$ satisfies the functional equation
\begin{equation}\label{srebfh}
	f^{\alpha} =\mathcal{T}(f^{\alpha})=\sum_{{\bf i}\in\{1,2,\ldots,N\}^k} g\cdot\chi_{\Delta_{\bf i}}+\sum_{{\bf i}\in\{1,2,\ldots,N\}^k}\alpha_{\bf i}\big(f^{\alpha}-b\big)\circ L_{\bf i}^{-1}\cdot\chi_{\Delta_{\bf i}}.  
\end{equation}
Consider the IFS $\mathscr{W}=\left\{\big(\Delta\times \mathbb{R};~W_{\bf i}\big):~{\bf i}\in\{1,2,\ldots,N\}^k\right\}$, where the maps
  $W_{\bf i}:\Delta\times \mathbb{R}\to \Delta_{\bf i}\times \mathbb{R}$ are given by
	\begin{equation}\label{wisfr}
		W_{\bf i}(x,y)=(L_{\bf i}(x),F_{\bf i}(x,y)).
	\end{equation}
 It thus follows that this IFS is hyperbolic, hence it has a unique attractor $G$. 
 \begin{theorem}
The attractor of the above IFS  $\left\{\big(\Delta\times \mathbb{R}; W_{\bf i}\big):~{\bf i}\in\{1,2,\ldots,N\}^k\right\}$ is the graph of the fractal function $f^{\alpha}$.
 \end{theorem}
\begin{proof}
Let	$G(f^{\alpha}):=\{\big(x,f^{\alpha}(x)\big):~x\in\Delta\}$. Then 
\begin{align*}
	\bigcup_{{\bf i}\in\{1,2,\ldots,N\}^k}W_{{\bf i}}\left(G(f^{\alpha})\right)&=\bigcup_{{\bf i}\in\{1,2,\ldots,N\}^k}\left\{W_{{\bf i}}\big(x,f^{\alpha}(x)\big):~ x\in\Delta\right\}\\
	&=\bigcup_{{\bf i}\in\{1,2,\ldots,N\}^k}\left\{\left(L_{\bf i}(x),F_{\bf i}\big(x,f^{\alpha}(x)\big)\right):~ x\in\Delta \right\}.
\end{align*}
Now, from (\ref{frabu}), (\ref{lasmahsy}) and (\ref{srebfh}), we get
\begin{align*}
	F_{\bf i}\big(x,f^{\alpha}(x)\big)&=g\big(L_{\bf i}(x)\big)-\alpha_{\bf i}b(x)+\alpha_{\bf i}f^{\alpha}(x)\\
	&=g\big(L_{\bf i}(x)\big)+\alpha_{\bf i}\big(f^{\alpha}-b\big)(x)=f^{\alpha}\big(L_{\bf i}(x)\big). 
\end{align*}
Therefore,
\begin{align*}
	\bigcup_{{\bf i}\in\{1,2,\ldots,N\}^k}W_{{\bf i}}\left(G(f^{\alpha})\right)&=\bigcup_{{\bf i}\in\{1,2,\ldots,N\}^k}\left\{\left(L_{\bf i}(x),f^{\alpha}\big(L_{\bf i}(x)\big)\right):~ x\in\Delta \right\}\\
	&=\bigcup_{{\bf i}\in\{1,2,\ldots,N\}^k}\left\{\big(x,f^{\alpha}(x)\big):~ x\in\Delta_{\bf i} \right\}\\
	&=\{\big(x,f^{\alpha}(x)\big):~x\in\Delta\}=G(f^{\alpha}).
\end{align*}
This shows that $G(f^{\alpha})$ is the attractor of the IFS, and hence by uniqueness $G=G(f^{\alpha})$. This completes the proof.
\end{proof}

The following examples illustrate the construction of non-affine fractal hypersurfaces for values of $k=1$ and $2$ respectively.

\begin{example}\label{1stillus}
	 For simplicity, we consider a $2$-simplex $\Delta$ in $\mathbb{R}^2$ and the classical function $g:\Delta\to\mathbb{R}$ given by (	see Figure \ref{fig1})
 	\begin{equation*}
 		g(x,y)=5+x^3 + y^2+\sin{2\pi x}\sin{2\pi y}.
 	\end{equation*}
 
 	\begin{figure}[ht!]
 		\centering
 		\includegraphics[width=10cm, height=7cm]{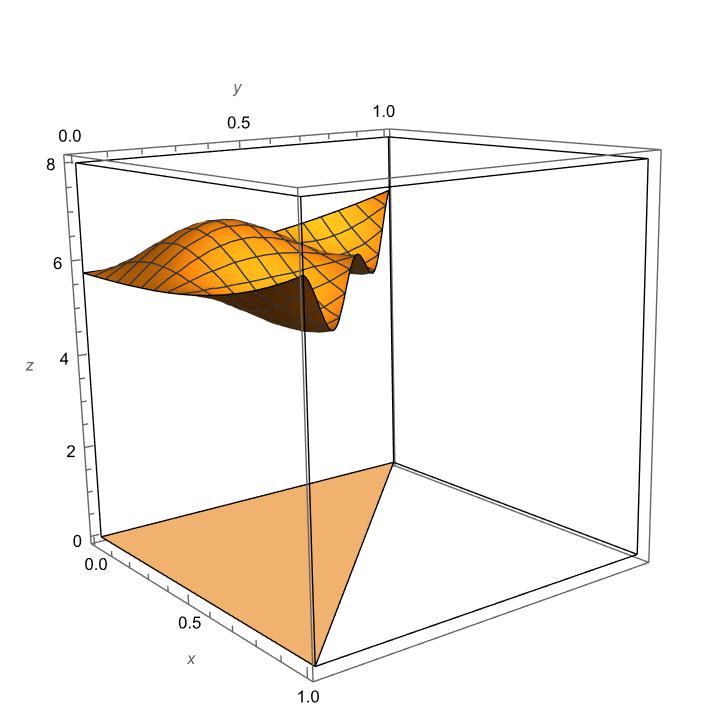}
 		\caption{Graphical representation of the classical function $g$.} \label{fig1}
 	\end{figure}
 	 For $k=1$, the vertex set is $\big\{(0,0),(\frac{1}{2},0),(1,0),(\frac{1}{2},\frac{1}{2}),(0,1),(0,\frac{1}{2})\big\}$. We consider the scale factors $\alpha_1=\alpha_2=\frac{4}{5}$ and $\alpha_3=\alpha_4=\frac{3}{4}$ and the base function $b:\Delta\to\mathbb{R}$ given by $b(x,y)=5+x^3 + y^2$. Then Figure~\ref{fig3} represents the corresponding non-affine fractal hypersurface.
 		\begin{figure}[ht]
 			\centering
 			\includegraphics[width=10cm, height=7cm]{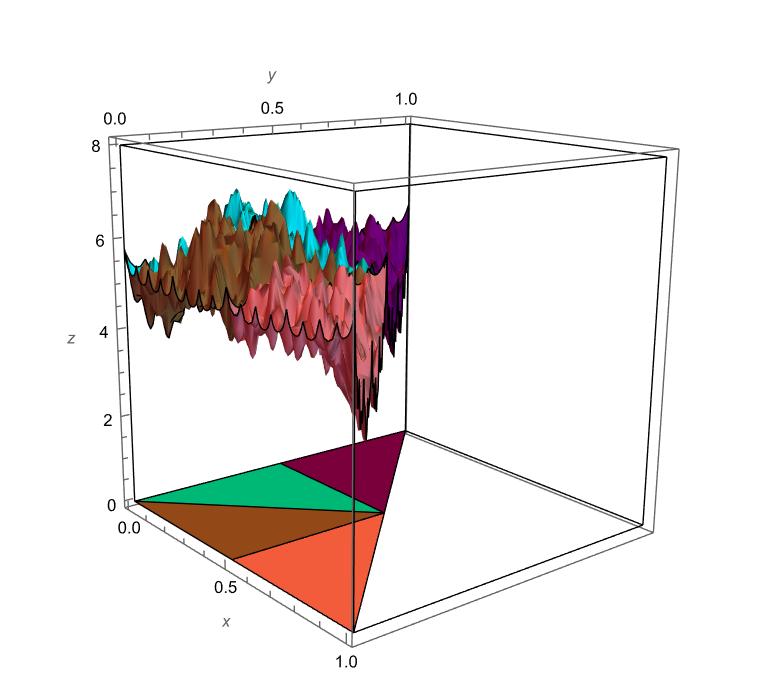}
 			\caption{Graphical representation of the non-affine fractal hypersurface for $k=1$.} \label{fig3}
 		\end{figure}
 	\end{example}
 	\begin{example} 
 		We consider a $2$-simplex $\Delta$ in $\mathbb{R}^2$ and the classical function $g$ and the base function $b$ given in Example~\ref{1stillus}.
 		 For $k=2$, vertex set is\\ $\big\{(0,0),(\frac{1}{4},0),(\frac{1}{2},0),(\frac{3}{4},0),(1,0),(\frac{3}{4},\frac{1}{4}),(\frac{1}{2},\frac{1}{2}),(\frac{1}{4},\frac{3}{4}),(0,1),(0,\frac{3}{4}),(0,\frac{1}{2}),(0,\frac{1}{4}),(\frac{1}{4},\frac{1}{4}),(\frac{1}{2},\frac{1}{4}),(\frac{1}{4},\frac{1}{2})\big\}$. For $i,j\in\{1,2,3,4\}$, consider the scale factors $\alpha_{ij}=\alpha_i\cdot\alpha_j$. Then  Figure~\ref{fig4} represents the corresponding non-affine fractal hypersurface.
 	\begin{figure}[ht]
 	\centering
 	\includegraphics[width=10cm, height=8cm]{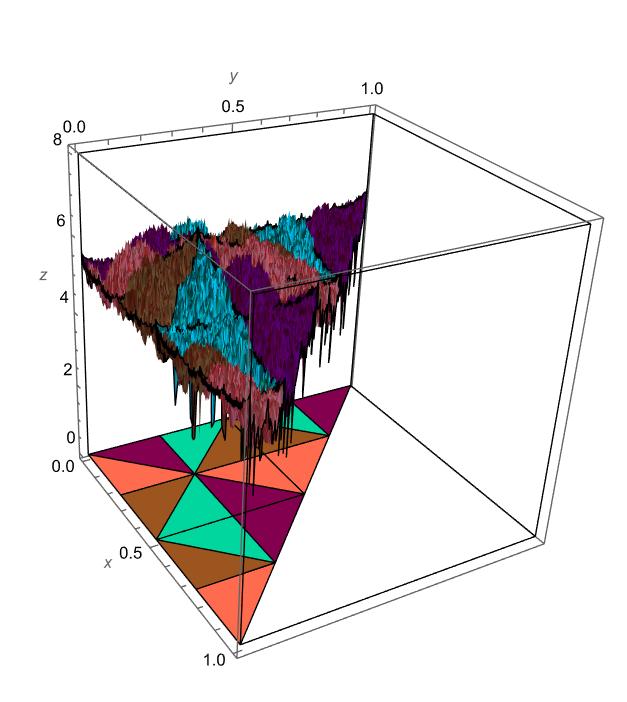}
 	\caption{Graphical representation of the non-affine fractal hypersurface for $k=2$.} \label{fig4}
 \end{figure}    
 \end{example}
 The following result can be found in \cite{Hutchinson1981}.
 \begin{proposition}\label{inmppp}
 	Let $(p_1,p_2,\ldots,p_N)$ be a given probability vector and $\{X; W_i:~i=1,2,\ldots, N\}$ be a hyperbolic IFS. Then there exists a unique Borel probability measure $\mu$ supported on its attractor such that
 	\begin{equation}
 		\mu=\sum_{ i=1}^Np_i\mu\circ W_i^{-1}.
 	\end{equation}
 \end{proposition}
 Let $(p_1,p_2,\ldots,p_N)$ be a given probability vector and for ${\bf i}\in\{1,2,\ldots,N\}^k$, let \\$p_{\bf i}=p_{i_1}\cdot p_{i_2}\cdots p_{i_k}$, $i_j\in\{1,\ldots,N\}$ for $j=1,\ldots,k$. Let $\mu$ and $\mu_{\alpha}$ be the invariant probability measures with probability vector $\big(p_{\bf i}\big)_{{\bf i}\in\{1,2,\ldots,N\}^k}$, generated by the IFSs $\left\{\big(\Delta; L_{\bf i}\big):~{\bf i}\in\{1,2,\ldots,N\}^k\right\}$ and $\left\{\big(\Delta\times \mathbb{R}; W_{\bf i}\big):~{\bf i}\in\{1,2,\ldots,N\}^k\right\}$ given in (\ref{baseifs}) and (\ref{wisfr}) respectively. Then the support of $\mu$ is $\Delta$ and the support of $\mu_{\alpha}$ is $G(f^{\alpha})$ respectively.
We end this section by providing a relation between $\mu$ and $\mu_{\alpha}$ in the following theorem.
\begin{theorem}
    Let $S : \Delta \to G(f^{\alpha})$ be the homeomorphism given by $S(x) = \big(x, f^{\alpha}(x)\big)$, for $x\in\Delta$. Then 
   \begin{equation*}
       \mu(E)=\mu_{\alpha}\big(S(E)\big)
   \end{equation*}
   for all Borel subsets $E$ of $\Delta$.
 \end{theorem}
 \begin{proof}
 Let $\mathscr{B}(G(f^{\alpha}))$ and $\mathscr{B}(\Delta)$ be the spaces of Borel probability measures supported on $G(f^{\alpha})$ and $\Delta$ respectively. Define the operator 
 \begin{align}\label{ptetf}
     \Phi:\mathscr{B}(G(f^{\alpha}))&\to \mathscr{B}(\Delta)\\
   \nonumber  ~~~~~~~\nu&\longrightarrow\Phi\nu
 \end{align}
 such that $\Phi\nu(E)=\nu\big(S(E)\big)$ for all Borel subsets $E$ of $\Delta$. Now, from Proposition~\ref{inmppp}, we get
 \begin{equation*}
        \mu_{\alpha}=\sum_{ {\bf i}\in\{1,2,\ldots,N\}^k}p_{\bf i}\mu_{\alpha}\circ W_{\bf i}^{-1}.
    \end{equation*}
    Therefore, for a Borel subset $E$ of $\Delta$
    \begin{align}\label{pestyh}
    \nonumber  \Phi\mu_{\alpha}(E)&=\mu_{\alpha}\big(S(E)\big)\\
     &=\sum_{  {\bf i}\in\{1,2,\ldots,N\}^k}p_{\bf i}\mu_{\alpha}\circ W_{\bf i}^{-1}\big(S(E)\big).
    \end{align}
    Now, for $x\in E$,
    \begin{align*}
        W_{\bf i}\big(L_{\bf i}^{-1}(x),f^{\alpha}(L_{\bf i}^{-1}(x))\big)&=\big(L_{\bf i}(L_{\bf i}^{-1}(x)),F_{\bf i}(L_{\bf i}^{-1}(x)),f^{\alpha}(L_{\bf i}^{-1}(x))\big),\quad\mbox{using (\ref{wisfr})}\\
        &=\big(x,\mathcal{T}f^{\alpha}(x)\big)\\
        &=\big(x,f^{\alpha}(x)\big),\quad\mbox{using (\ref{srebfh})}\\
        &=S(x).
    \end{align*}
    This shows that $W_{\bf i}^{-1}\big(S(x)\big)=\big(L_{\bf i}^{-1}(x),f^{\alpha}(L_{\bf i}^{-1}(x))\big)=S\big(L_{\bf i}^{-1}(x)\big)$, $x\in E$.
    Therefore, $W_{\bf i}^{-1}\big(S(E)\big)=S\big(L_{\bf i}^{-1}(E)\big)$. Hence from (\ref{pestyh}), we get
    \begin{align*}
        \Phi\mu_{\alpha}(E)&=\sum_{ {\bf i}\in\{1,2,\ldots,N\}^k}p_{\bf i}\mu_{\alpha}\circ S\big(L_{\bf i}^{-1}(E)\big)\\
        &=\sum_{ {\bf i}\in\{1,2,\ldots,N\}^k}p_{\bf i}\Phi\mu_{\alpha}\circ L_{\bf i}^{-1}(E),\quad \mbox{using (\ref{ptetf})}.
    \end{align*}
    But $\mu$ is the unique probability measure supported on $\Delta$ such that
    \begin{align*}
        \mu(E)=\sum_{ {\bf i}\in\{1,2,\ldots,N\}^k}p_{\bf i}\mu\circ L_{\bf i}^{-1}(E).
    \end{align*}
    Therefore, $\mu(E)=\Phi\mu_{\alpha}(E)=\mu_{\alpha}\big(S(E)\big)$. This completes the proof.
 \end{proof}
 \section{Dimension results} 
In this section, we explore the $\beta$-oscillation space (see \cite{Verma2023}) and derive bounds for the fractal dimension of the graph of a non-affine multivariate fractal function. Additionally, we provide an upper bound for the Hausdorff dimension of the invariant probability measure supported on the graph.\\
 
 For $N\geq 2$ and ${\bf i}=(i_1,i_2,\ldots,i_k)\in\{1,2,\ldots,N\}^k$, recall the definition of the sets $\Delta_{{\bf i}}:=L_{{\bf i}}(\Delta)=L_{i_1}\circ L_{i_2}\circ\cdots\circ L_{i_k}(\Delta)$. The maximal range of a function $g:\Delta_{\bf i}\to\mathbb{R}$ over the set $\Delta_{\bf i}$ is defined by
 \begin{align*}
     \mathcal{R}_g[{\Delta_{\bf i}}]=\sup_{x_1,x_2\in\Delta_{\bf i}}\vert g(x_1)-g(x_2)\vert
 \end{align*}
 and the total oscillation of order $k$ is defined by
 \begin{align*}
     \mathcal{R}(k,g)=\sum_{{\bf i}\in\{1,2,\ldots,N\}^k}\mathcal{R}_g[{\Delta_{\bf i}}].
 \end{align*}
 For $0\leq \beta\leq2$, define the function space
 \begin{align}
     \mathscr{R}^{\beta}(\Delta):=\left\{g:\Delta\to\mathbb{R}:~g ~\mbox{is continuous and}~\sup_{k\in \mathbb{N}}\frac{\mathcal{R}(k,g)}{N^{k(2-\beta)}}<\infty\right\}.
 \end{align}
 $\mathscr{R}^{\beta}(\Delta)$ is called the $\beta$-\textbf{oscillation space}. Here, we obtain the following results.
 \begin{proposition}\label{doaft}
     For all $g\in \mathscr{R}^{\beta}(\Delta)$, the following holds:
     \begin{enumerate}
         \item If $0\leq\beta\leq 1$, then 
         \begin{align*}
             2\leq \dim_HG(g)\leq \dim_BG(g)\leq 3-\beta.
         \end{align*}
         \item If $1<\beta\leq 2$, then 
         \begin{align*}
            \dim_HG(g)=\dim_BG(g)=2.
         \end{align*}
     \end{enumerate}
 \end{proposition}
 \begin{proof}
     Let $g\in \mathscr{R}^{\beta}(\Delta)$. Then by definition of $\beta$-oscillation space, there exists a real $M>0$ such that 
     \begin{align}\label{sbfav}
         \sup_{k\in \mathbb{N}}\frac{\mathcal{R}(k,g)}{N^{k(2-\beta)}}\leq M.
     \end{align}
     For a fixed $k\in\mathbb{N}$, let $\delta=\frac{1}{N^k}$. Then by the continuity of $g$, the number of mesh-prism of side lengths $\delta$ in the column above the set $\Delta_{\bf i}$ that intersects $G(g)$ is at least $\frac{\mathcal{R}_g[{\Delta_{\bf i}}]}{\delta}$ and at most $2+\frac{\mathcal{R}_g[{\Delta_{\bf i}}]}{\delta}$. Summing over all such set $\Delta_{\bf i}$, we get
     \begin{align}\label{gcnf}
        N^k\sum_{{\bf i}\in\{1,2,\ldots,N\}^k}\mathcal{R}_g[{\Delta_{\bf i}}]\leq N_{\delta}\big(G(g)\big)\leq 2N^k+N^k\sum_{{\bf i}\in\{1,2,\ldots,N\}^k}\mathcal{R}_g[{\Delta_{\bf i}}].
     \end{align}
    Now since $N\geq 2$, using (\ref{sbfav}) and (\ref{gcnf}), we get
    \begin{align*}
        N_{\delta}\big(G(g)\big)&\leq N^{2k}+N^k M N^{k(2-\beta)}\\
        &=N^{k(3-\beta)}\big(M+N^{k(\beta-1)}\big).
    \end{align*}
    Therefore,
    \begin{align*}
        \dim_BG(g)=\lim_{k\to\infty}\frac{\log N_{\delta}\big(G(g)\big)}{\log N^k}\leq 3-\beta+\lim_{k\to\infty}\frac{\log\big(M+N^{k(\beta-1)}\big)}{\log N^k}.
        \end{align*}
        {\bf Case 1.} If $0\leq\beta\leq 1$, then $N^{k(\beta-1)}\to 0$, which implies that
        \begin{align*}
            \lim_{k\to\infty}\frac{\log\big(M+N^{k(\beta-1)}\big)}{\log N^k}=0.
        \end{align*}
        Also, since $\dim_H\Delta=2$ and $G(g)$ is the graph of the continuous function $g:\Delta\to\mathbb{R}$, it follows that $\dim_HG(g)\geq 2$.
        Therefore, 
         \begin{align*}
             2\leq \dim_HG(g)\leq \dim_BG(g)\leq 3-\beta.
         \end{align*}
        {\bf Case 2.} If $1<\beta\leq 2$, then $\frac{1}{N^{k(\beta-1)}}\to 0$, which implies that
        \begin{align*}
             \lim_{k\to\infty}\frac{\log\big(M+N^{k(\beta-1)}\big)}{\log N^k}= \lim_{k\to\infty}\frac{\log N^{k(\beta-1)}\left(\frac{M}{N^{k(\beta-1)}}+1\right)}{\log N^k}=\beta-1.
        \end{align*}
        This shows that
      \begin{align*}
             2\leq \dim_HG(g)\leq \dim_BG(g)\leq 3-\beta+\beta-1=2.
         \end{align*} 
         Therefore, 
       \begin{align*}
              \dim_HG(g)=\dim_BG(g)=2.
         \end{align*}
         This completes the proof.
 \end{proof}
In Section~\ref{conafhsfr}, we considered the original function $g$ and the base function $b$ in $\mathcal{C}_{0}(\Delta)$. From this point onwards, we consider the original function  $g$ and the base function $b$ in $\mathscr{R}^{\beta}(\Delta)$, and take $f^\alpha$ to be the corresponding non-affine fractal function. In the following theorem, we estimate the fractal dimension of $G(f^\alpha)$ and provide bounds for the Hausdorff dimension of the measure $\mu_{\alpha}$.
 \begin{theorem}\label{hypdimthem}
  Suppose, for a fixed $k\in\mathbb{N}$, that  $\max\left\{ \alpha^k_{\infty}, ~\frac{ \alpha^k_{\infty}}{N^{k(1-\beta)}}\right\} <1$. Then the non-affine fractal function $f^\alpha$ lies in the  $\beta$-oscillation space $\mathscr{R}^{\beta}(\Delta)$. Furthermore, the following holds:
  \begin{enumerate}
  	\item If $0\leq\beta\leq 1$, then 
  	\begin{align*}
  		2\leq \dim_HG(f^\alpha)\leq \dim_BG(f^\alpha)\leq 3-\beta,~\mbox{and}~\dim_H\mu_{\alpha} \leq 3-\beta. 
  	\end{align*}
  	\item If $1<\beta\leq 2$, then 
  	\begin{align*}
  	\dim_HG(f^\alpha)=\dim_BG(f^\alpha)=2,~\mbox{and}~\dim_H\mu_{\alpha} \leq 2. 
  	\end{align*}
  \end{enumerate} 
 \end{theorem}
 To prove the above theorem, we will first establish the following results:\\
 For $g\in \mathscr{R}^{\beta}(\Delta)$, define a norm 
 \begin{equation}\label{betanorm}
 \lVert g\rVert_{\mathscr{R}^{\beta}(\Delta)}:=\lVert g\rVert_{\infty,\Delta}+ \sup_{k\in \mathbb{N}}\frac{\mathcal{R}(k,g)}{N^{k(2-\beta)}}.
 \end{equation}
 It is easily shown that this defines a norm. Indeed,
 \begin{enumerate}
 	\item $\lVert g\rVert_{\mathscr{R}^{\beta}(\Delta)}=0$, if and only if $\lVert g\rVert_{\infty,\Delta}=0$, if and only if $g=0$.
 	\item Now, for $c\in\mathbb{R}$ and $g\in \mathscr{R}^{\beta}(\Delta)$,  $\mathcal{R}_{cg}[{\Delta_{\bf i}}]=\lvert c\rvert \mathcal{R}_{g}[{\Delta_{\bf i}}]$, therefore, $\mathcal{R}(k,cg)=\lvert c\rvert\mathcal{R}(k,g)$. Hence $\lVert cg\rVert_{\mathscr{R}^{\beta}(\Delta)}= \lvert c\rvert\lVert g\rVert_{\mathscr{R}^{\beta}(\Delta)}$.
 	\item For $f,g\in \mathscr{R}^{\beta}(\Delta)$, $\mathcal{R}_{f+g}[{\Delta_{\bf i}}]\leq \mathcal{R}_{f}[{\Delta_{\bf i}}]+\mathcal{R}_{g}[{\Delta_{\bf i}}]$. Therefore, $\mathcal{R}(k,f+g)\leq \mathcal{R}(k,f)+\mathcal{R}(k,g)$. Hence from (\ref{betanorm}), we get $\lVert f+g\rVert_{\mathscr{R}^{\beta}(\Delta)}\leq\lVert f\rVert_{\infty,\Delta}+\lVert g\rVert_{\infty,\Delta}+ \sup_{k\in \mathbb{N}}\frac{\mathcal{R}(k,f)}{N^{k(2-\beta)}}+\sup_{k\in \mathbb{N}}\frac{\mathcal{R}(k,g)}{N^{k(2-\beta)}}.$ Therefore, $\lVert f+g\rVert_{\mathscr{R}^{\beta}(\Delta)}\leq \lVert f\rVert_{\mathscr{R}^{\beta}(\Delta)}+\lVert g\rVert_{\mathscr{R}^{\beta}(\Delta)}$.
 \end{enumerate}
 \begin{lemma}
 	The space $\big(\mathscr{R}^{\beta}(\Delta),\lVert \cdot\rVert_{\mathscr{R}^{\beta}(\Delta)}\big)$ is a Banach space.
 \end{lemma}
 \begin{proof}
 	Let $\big(f_n\big)_{n\in\mathbb{N}}$ be a Cauchy sequence in $\big(\mathscr{R}^{\beta}(\Delta),\lVert \cdot\rVert_{\mathscr{R}^{\beta}(\Delta)}\big)$. Then $\big(f_n\big)_{n\in\mathbb{N}}$ is a Cauchy sequence in $\big(\mathcal{C}(\Delta),\lVert \cdot\rVert_{\infty,\Delta}\big)$, and hence converges to a continuous function $f$. Our first claim is $\mathcal{R}(k,f_n)\to \mathcal{R}(k,f)$ as $n\to\infty$. Since $f_n\to f$ uniformly, it follows that
 	\begin{align*}
 		\lvert f_n(x_1)-f_n(x_2)\rvert\to \lvert f(x_1)-f(x_2)\rvert\quad\mbox{for all}~ x_1,x_2\in \Delta_{\bf i},~{\bf i}\in\{1,2,\ldots,N\}^k.
 	\end{align*}
 	This shows that 
 	\begin{align*}
 	\sup_{x_1,x_2\in\Delta_{\bf i}}\lvert f_n(x_1)-f_n(x_2)\rvert\to \sup_{x_1,x_2\in\Delta_{\bf i}}\lvert f(x_1)-f(x_2)\rvert.
 	\end{align*}
 	Therefore, $\mathcal{R}_{f_n}[{\Delta_{\bf i}}]\to\mathcal{R}_{f}[{\Delta_{\bf i}}]$. Hence $\mathcal{R}(k,f_n)\to \mathcal{R}(k,f)$. Therefore, $\lVert f_n\rVert_{\mathscr{R}^{\beta}(\Delta)}=\lVert f_n\rVert_{\infty,\Delta}+ \sup_{k\in \mathbb{N}}\frac{\mathcal{R}(k,f_n)}{N^{k(2-\beta)}}\to \lVert f\rVert_{\infty,\Delta}+ \sup_{k\in \mathbb{N}}\frac{\mathcal{R}(k,f)}{N^{k(2-\beta)}}=\lVert f\rVert_{\mathscr{R}^{\beta}(\Delta)}$. This completes the proof.
 \end{proof}
 \begin{proof}[Proof of Theorem \ref{hypdimthem}]
 	Let $\mathscr{R}^{\beta}_{f^\alpha}(\Delta):=\left\{f\in\mathscr{R}^{\beta}(\Delta):~f\rvert_{Z_k}=f^\alpha\rvert_{Z_k}\right\}$, where $Z_k=\bigcup_{{\bf i}\in\{1,2,\ldots,N\}^k}V_{\bf i}$. It is easy to show that $\mathscr{R}^{\beta}_{f^\alpha}(\Delta)$ is a closed subset of $\mathscr{R}^{\beta}(\Delta)$, and so it is complete with respect to the metric induced by the norm $\lVert \cdot\rVert_{\mathscr{R}^{\beta}(\Delta)}$. Define  the RB-operator  $\mathcal{T}:\mathscr{R}^{\beta}_{f^\alpha}(\Delta)\to \mathscr{R}^{\beta}_{f^\alpha}(\Delta)$, in an analogous way to (\ref{rboptr}), by
 	\begin{equation}\label{2rbtrf}
 		\mathcal{T}(f)=\sum_{{\bf i}\in\{1,2,\ldots,N\}^k} g\cdot\chi_{\Delta_{\bf i}}+\sum_{{\bf i}\in\{1,2,\ldots,N\}^k}\alpha_{\bf i}\big(f-b\big)\circ L_{\bf i}^{-1}\cdot\chi_{\Delta_{\bf i}},
 	\end{equation}
 	where for $k\in\mathbb{N}$, $\Delta=\bigcup_{{\bf i}\in\{1,2,\ldots,N\}^k}\Delta_{{\bf i}}$, $\alpha_{\bf i}=\alpha_{i_1}\alpha_{i_2}\cdots\alpha_{i_k}$ and $b\in \mathscr{R}^{\beta}_{f^\alpha}(\Delta)$ such that $b\neq g$. Let $v_k\in Z_k$, then $v_k\in\Delta_{\bf i}$ for some ${\bf i}\in\{1,2,\ldots,N\}^k$. Using (\ref{vstcnt}) in (\ref{2rbtrf}), we get
 	\begin{align*}
 		\mathcal{T}(f)(v_k)&=g(v_k)+\alpha_i\big(f-b\big)(l(v_k))\\
 		&=z_{v_k}+\alpha_i\big(z_{l_k(v_k)}-z_{l_k(v_k)}\big)=f^\alpha(v_k).
 	\end{align*}
 	Therefore, $\mathcal{T}(f)\rvert_{Z_k}=f^\alpha\rvert_{Z_k}$.
 	Also, for $x\in E_{{\bf ij}}=\Delta_{\bf i}\cap\Delta_{\bf j}$, from (\ref{2rbtrf}) and (\ref{edgrtg}), we get
 	\begin{equation*}
 		\mathcal{T}(f)(x)=g(x).
 	\end{equation*}
 	This shows that $\mathcal{T}$ is well defined. To check contractivity of $\mathcal{T}$, let $f_1,f_2\in \mathscr{R}^{\beta}_{f^\alpha}(\Delta)$. Then 
 	\begin{align}\label{tgsue}
 \nonumber		 \lVert \mathcal{T}f_1-\mathcal{T}f_2\rVert_{\mathscr{R}^{\beta}(\Delta)}&=\lVert \mathcal{T}f_1-\mathcal{T}f_2\rVert_{\infty,\Delta}+ \sup_{m\in \mathbb{N}}\frac{\mathcal{R}(m,\mathcal{T}f_1-\mathcal{T}f_2)}{N^{m(2-\beta)}}\\
 		 &=\lVert \mathcal{T}f_1-\mathcal{T}f_2\rVert_{\infty,\Delta}+\sup_{m\in \mathbb{N}}\frac{\sum_{{\bf i}\in\{1,2,\ldots,N\}^m}\mathcal{R}_{\mathcal{T}f_1-\mathcal{T}f_2}[{\Delta_{\bf i}}]}{N^{m(2-\beta)}}.
 	\end{align}
 	First, we estimate the quantity $\sum_{{\bf i}\in\{1,2,\ldots,N\}^m}\mathcal{R}_{\mathcal{T}f_1-\mathcal{T}f_2}[{\Delta_{\bf i}}]$. For any $m>k$, the word ${\bf i}\in\{1,2,\ldots,N\}^m$ may be expressed as ${\bf i}={\bf (j_1,j_2)}$, where ${\bf j_1}=(j_1,j_2,\ldots,j_k)\in\{1,2,\ldots,N\}^k$ and ${\bf j_2}=(j_{k+1},j_{k+2},\ldots,j_{m})\in\{1,2,\ldots,N\}^{m-k}$. Now, using the expression of ${\bf i}={\bf (j_1,j_2)}$, we may write
 	\begin{align*}
 		\sum&_{{\bf i}\in\{1,2,\ldots,N\}^m}\mathcal{R}_{\mathcal{T}f_1-\mathcal{T}f_2}[{\Delta_{\bf i}}]\\
 		&=\sum_{{\bf i}\in\{1,2,\ldots,N\}^m}\sup_{x_1,x_2\in\Delta_{\bf i}=L_{\bf i}(\Delta)}\vert \big(\mathcal{T}f_1-\mathcal{T}f_2\big)(x_1)-\big(\mathcal{T}f_1-\mathcal{T}f_2\big)(x_2)\vert\\
 		&=\sum_{{\bf i}\in\{1,2,\ldots,N\}^m}\sup_{x_1,x_2\in\Delta}\vert \big(\mathcal{T}f_1-\mathcal{T}f_2\big)\circ L_{\bf i}(x_1)-\big(\mathcal{T}f_1-\mathcal{T}f_2\big)\circ L_{\bf i}(x_2)\vert\\
 		&=\sum_{{\bf j_2}\in\{1,2,\ldots,N\}^{m-k}}\sum_{{\bf j_1}\in\{1,2,\ldots,N\}^{k}}\sup_{x_1,x_2\in\Delta}\vert \big(\mathcal{T}f_1-\mathcal{T}f_2\big)\circ L_{\bf j_1}\circ L_{\bf j_2}(x_1)-\big(\mathcal{T}f_1-\mathcal{T}f_2\big)\circ  L_{\bf j_1}\circ L_{\bf j_2}(x_2)\vert.
 	\end{align*}
 Now, for the partition $\Delta=\bigcup_{{\bf j_1}\in\{1,2,\ldots,N\}^{k}}\Delta_{{\bf j_1}}$, using (\ref{2rbtrf}) in the last expression, we get
 \begin{align*}
 	&\sum_{{\bf i}\in\{1,2,\ldots,N\}^m}\mathcal{R}_{\mathcal{T}f_1-\mathcal{T}f_2}[{\Delta_{\bf i}}]\\
	&=\sum_{{\bf j_2}\in\{1,2,\ldots,N\}^{m-k}}\sum_{{\bf j_1}\in\{1,2,\ldots,N\}^{k}}\lvert\alpha_{\bf j_1}\rvert\sup_{x_1,x_2\in\Delta}\vert\big(f_1-f_2\big)\circ L^{-1}_{\bf j_1}\circ L_{\bf j_1}\circ L_{\bf j_2}(x_1)-\big(f_1-f_2\big)\circ L^{-1}_{\bf j_1}\circ  L_{\bf j_1}\circ L_{\bf j_2}(x_2)\vert\\
	&\leq\alpha^k_{\infty} \sum_{{\bf j_2}\in\{1,2,\ldots,N\}^{m-k}}\sum_{{\bf j_1}\in\{1,2,\ldots,N\}^{k}}\sup_{x_1,x_2\in\Delta}\vert\big(f_1-f_2\big)\circ L_{\bf j_2}(x_1)-\big(f_1-f_2\big)\circ L_{\bf j_2}(x_2)\vert\\
		&=N^k\alpha^k_{\infty} \sum_{{\bf j_2}\in\{1,2,\ldots,N\}^{m-k}}\sup_{x_1,x_2\in\Delta}\vert\big(f_1-f_2\big)\circ L_{\bf j_2}(x_1)-\big(f_1-f_2\big)\circ L_{\bf j_2}(x_2)\vert\\
		&=(N\alpha_{\infty})^k\sum_{{\bf j_2}\in\{1,2,\ldots,N\}^{m-k}}\sup_{x_1,x_2\in L_{{\bf j_2}}(\Delta)=\Delta_{{\bf j_2}}}\vert\big(f_1-f_2\big)(x_1)-\big(f_1-f_2\big)(x_2)\vert\\
	&=(N\alpha_{\infty})^k\mathcal{R}(m-k,f_1-f_2).
 \end{align*} 
Hence for the partition $\Delta=\bigcup_{{\bf j_1}\in\{1,2,\ldots,N\}^{k}}\Delta_{{\bf j_1}}$, using the above expression and (\ref{rbcntv}) in (\ref{tgsue}), we get
\begin{align*}
	 \lVert \mathcal{T}f_1-\mathcal{T}f_2\rVert_{\mathscr{R}^{\beta}(\Delta)}&\leq \alpha^k_{\infty}\lVert f_1-f_2\rVert_{\infty,\Delta}+(N\alpha_{\infty})^k\sup_{m\in \mathbb{N}}\frac{\mathcal{R}(m-k,f_1-f_2)}{N^{m(2-\beta)}}\\
	 &=\alpha^k_{\infty}\lVert f_1-f_2\rVert_{\infty,\Delta}+\frac{(N\alpha_{\infty})^k}{N^{k(2-\beta)}}\sup_{m\in \mathbb{N}}\frac{\mathcal{R}(m-k,f_1-f_2)}{N^{(m-k)(2-\beta)}}\\
	 &\leq\max\left\{\alpha^k_{\infty},\frac{\alpha^k_{\infty}}{N^{k(1-\beta)}}\right\}\left(\lVert f_1-f_2\rVert_{\infty,\Delta}+\sup_{m\in \mathbb{N},m>k}\frac{\mathcal{R}(m,f_1-f_2)}{N^{m(2-\beta)}}\right)\\
	 &=c~\lVert f_1-f_2\rVert_{\mathscr{R}^{\beta}(\Delta)},
\end{align*}
where $c=\max\left\{\alpha^k_{\infty},\frac{\alpha^k_{\infty}}{N^{k(1-\beta)}}\right\} <1$. This shows that $\mathcal{T}$ is a contraction map on $\big(\mathscr{R}^{\beta}_{f^\alpha}(\Delta),\lVert \cdot\rVert_{\mathscr{R}^{\beta}(\Delta)}\big)$. Hence, by the Banach fixed point theorem, $\mathcal{T}$ has a fixed point $f^*\in\mathscr{R}^{\beta}_{f^\alpha}(\Delta)$. Also $\mathcal{T}(f^*)\rvert_{Z_k}=f^\alpha\rvert_{Z_k}$, and it satisfies the functional equation (\ref{srebfh}). Hence by uniqueness we conclude that $ f^\alpha=f^*\in\mathscr{R}^{\beta}(\Delta)$.\\
 Now, using Proposition \ref{doaft}, we get	
	\begin{align*}
		2\leq \dim_HG(f^\alpha)\leq \dim_BG(f^\alpha)\leq 3-\beta, ~\mbox{if}~ 0\leq\beta\leq 1
	\end{align*} 
	and
	\begin{align*}
	\dim_HG(f^\alpha)=\dim_BG(f^\alpha)=2,~\mbox{if}~1<\beta\leq 2.
	\end{align*}
 Also, since $\mu_{\alpha}$ is the probability measure with support $G(f^\alpha)$ generated by the IFS \\
 $\left\{\big(\Delta\times \mathbb{R}; W_{\bf i}\big):~{\bf i}\in\{1,2,\ldots,N\}^k\right\}$, it follows that $\mu_{\alpha}\big(G(f^\alpha)\big)>0$. Hence by (\ref{mmdfg}), we get
 \begin{align*}
 	\dim_H\mu_{\alpha}\leq \dim_HG(f^\alpha).
 \end{align*} 
Therefore,
\begin{enumerate}
	\item if $0\leq\beta\leq 1$, then $\dim_H\mu_{\alpha} \leq 3-\beta$, and 
	\item  if $1<\beta\leq 2$, then $\dim_H\mu_{\alpha} \leq 2$.
\end{enumerate}
This completes the proof.
 \end{proof}
\section*{Conclusion}
In this article, we presented the construction of a non-affine hypersurface on an $n$-simplex in $\mathbb{R}^n$. We also estimated the fractal dimension of the graph of this non-affine multivariate real-valued fractal function under certain conditions. Moreover, we estimated the upper bound of the Hausdorff dimension of the invariant probability measure supported on the graph of this fractal function. This investigation was conducted under the assumption of constant scale factors, within the context of Euclidean space. Future research directions may include considering variable scale factors and generalizing the results to $L^{p}$ spaces, thereby broadening the scope of applicability.\\

\textbf{Acknowledgments.} The first author would like to thank Dr. Md. Nasim Akhtar for insightful discussions and valuable suggestions. Also, extend sincere gratitude to Presidency University's Department of Mathematics for their support during this research. The second author acknowledges partial support from National Funding from FCT - Fundação para a Ciência e a Tecnologia, under the project: UIDB/04561/2020. 
\bibliographystyle{abbrv} 
\bibliography{Bibliography}	

\begin{thebibliography}{10}

\bibitem{Akhtar2022}
M.~N. Akhtar and A.~Hossain.
\newblock Stereographic metric and dimensions of fractals on the sphere.
\newblock {\em Rest. Math.}, 77(6):213, 2022.

\bibitem{Akhtar2016}
M.~N. Akhtar, M.~G.~P. Prasad, and M.~A. Navascu{\'e}s.
\newblock Box dimensions of $\alpha$-fractal functions.
\newblock {\em Fractals}, 24(03):1650037, 2016.

\bibitem{Barnsley1986}
M.~F. Barnsley.
\newblock Fractal functions and interpolation.
\newblock {\em Constr. Approx.}, 2(1):303--329, 1986.

\bibitem{Barnsley2014}
M.~F. Barnsley.
\newblock {\em Fractals Everywhere}.
\newblock Academic Press, New York, 2014.

\bibitem{Barnsley1993}
M.~F. Barnsley and L.~P. Hurd.
\newblock {\em Fractal image compression}.
\newblock AK Peters, Ltd., 1993.

\bibitem{barnsley2013developments}
M.~F. Barnsley and A.~Vince.
\newblock Developments in fractal geometry.
\newblock {\em Bull. Math. Sci.}, 3:299--348, 2013.

\bibitem{Bedford2014}
T.~Bedford, S.~V. B., and J.~S. Geronimo.
\newblock A topological separation condition for fractal attractors.
\newblock {\em J. Fractal Geom.}, 1(3):243--271, 2014.

\bibitem{Buescu2012}
J.~Buescu.
\newblock {\em Exotic Attractors: From Liapunov Stability to Riddled Basins},
  volume 153.
\newblock Birkh{\"a}user, 2012.

\bibitem{Buescu2019}
J.~Buescu and C.~Serpa.
\newblock Fractal and \mbox{Hausdorff} dimensions for systems of iterative
  functional equations.
\newblock {\em J. Math. Anal. Appl.}, 480(2):123--429, 2019.

\bibitem{Buescu2021}
J.~Buescu and C.~Serpa.
\newblock Compatibility conditions for systems of iterative functional
  equations with non-trivial contact sets.
\newblock {\em Result. Math.}, 76(2):19, 2021.
\newblock Id/No 68.

\bibitem{chand2015approximation}
A.~K.~B. Chand, S.~K. Katiyar, and P.~Viswanathan.
\newblock Approximation using hidden variable fractal interpolation function.
\newblock {\em J. Fractal Geom.}, 2(1):81--114, 2015.

\bibitem{Chand2015}
A.~K.~B. Chand, P.~Viswanathan, and N.~Vijender.
\newblock Bivariate shape preserving interpolation: a fractal-classical hybrid
  approach.
\newblock {\em Chaos, Solitons \& Fractals}, 81:330--344, 2015.

\bibitem{Coleman1992}
P.~H. Coleman and L.~Pietronero.
\newblock The fractal structure of the universe.
\newblock {\em Phys. Rep.}, 213(6):311--389, 1992.

\bibitem{Dalla2002}
L.~Dalla.
\newblock Bivariate fractal interpolation functions on grids.
\newblock {\em Fractals}, 10(1):53--58, 2002.

\bibitem{DAniello2016}
E.~D'Aniello and T.~H. Steele.
\newblock Attractors for iterated function systems.
\newblock {\em J. Fractal Geom.}, 3(2):95--117, 2016.

\bibitem{Eke2002}
A.~Eke, P.~Herman, L.~Kocsis, and L.~R. Kozak.
\newblock Fractal characterization of complexity in temporal physiological
  signals.
\newblock {\em Physiol. Meas.}, 23(1):R1, 2002.

\bibitem{Falconer2004}
K.~J. Falconer.
\newblock {\em Fractal Geometry: Mathematical Foundations and Applications}.
\newblock John Wiley \& Sons, England, 2004.

\bibitem{Fisher1994}
Y.~Fisher.
\newblock Fractal image compression.
\newblock {\em Fractals}, 2(03):347--361, 1994.

\bibitem{Hossain2023a}
A.~Hossain, M.~N. Akhtar, and M.~A. Navascu{\'e}s.
\newblock Fractal dimension of fractal functions on the real projective plane.
\newblock {\em Fractal fract.}, 7(7):510, 2023.

\bibitem{hossain2023fractal}
A.~Hossain, M.~N. Akhtar, and M.~A. Navascu{\'e}s.
\newblock Fractal interpolation on the real projective plane.
\newblock {\em Numer. Algorithms}, pages 1--26, 2023.

\bibitem{Hutchinson1981}
J.~E. Hutchinson.
\newblock Fractals and self similarity.
\newblock {\em Indiana Univ. Math. J.}, 30(5):713--747, 1981.

\bibitem{Jiang2023}
L.~Jiang and H.~J. Ruan.
\newblock Box dimension of generalized affine fractal interpolation functions.
\newblock {\em J. Fractal Geom.}, 10(3-4):279--302, 2023.

\bibitem{Liang2021}
Z.~Liang and H.~J. Ruan.
\newblock Construction and box dimension of recurrent fractal interpolation
  surfaces.
\newblock {\em J. Fractal Geom.}, 8(3):261--288, 2021.

\bibitem{Mandelbrot1982}
B.~Mandelbrot.
\newblock {\em The Fractal Geometry of Nature}.
\newblock W. H. Freeman and Co, New York, 1982.

\bibitem{Massopust1990}
P.~R. Massopust.
\newblock Fractal surfaces.
\newblock {\em J. Math. Anal. Appl.}, 151(1):275--290, 1990.

\bibitem{Massopust1997}
P.~R. Massopust.
\newblock Fractal functions and their applications.
\newblock {\em Chaos, Solit. Fractals}, 8(2):171--190, 1997.

\bibitem{massopust2024fractal}
P.~R. Massopust.
\newblock Fractal hypersurfaces, affine \mbox{Weyl} groups, and wavelet sets.
\newblock {\em J. Anal.}, 32(1):399--431, 2024.

\bibitem{Navascues2005}
M.~A. Navascu{\'e}s.
\newblock Fractal trigonometric approximation.
\newblock {\em Electron. Trans. Numer. Anal.}, 20:64--74, 2005.

\bibitem{Navascues2006}
M.~A. Navascu{\'e}s.
\newblock A fractal approximation to periodicity.
\newblock {\em Fractals}, 14(04):315--325, 2006.

\bibitem{Navascues2020}
M.~A. Navascu{\'e}s, R.~N. Mohapatra, and M.~N. Akhtar.
\newblock Construction of fractal surfaces.
\newblock {\em Fractals}, 28(02):2050033, 2020.

\bibitem{nussbaum2012positive}
R.~Nussbaum, A.~Priyadarshi, and S.~Verduyn~Lunel.
\newblock Positive operators and \mbox{Hausdorff} dimension of invariant sets.
\newblock {\em Trans. Amer. Math. Soc.}, 364(2):1029--1066, 2012.

\bibitem{Peters1994}
E.~E. Peters.
\newblock {\em Fractal Market Analysis: applying chaos theory to investment and
  economics}, volume~24.
\newblock John Wiley \& Sons, 1994.

\bibitem{Priyanka2021}
T.~M.~C. Priyanka and A.~Gowrisankar.
\newblock \mbox{Riemann--Liouville} fractional integral of non-affine fractal
  interpolation function and its fractional operator.
\newblock {\em EPJST}, 230(21):3789--3805, 2021.

\bibitem{sahu2020box}
A.~Sahu and A.~Priyadarshi.
\newblock On the box-counting dimension of graphs of harmonic functions on the
  \mbox{S}ierpi{\'n}ski gasket.
\newblock {\em J. Math. Anal. Appl.}, 487(2):124036, 2020.

\bibitem{Serpa2015}
C.~Serpa and J.~Buescu.
\newblock Non-uniqueness and exotic solutions of conjugacy equations.
\newblock {\em J. Difference Equ. Appl.}, 21(12):1147--1162, 2015.

\bibitem{Serpa2017}
C.~Serpa and J.~Buescu.
\newblock Constructive solutions for systems of iterative functional equations.
\newblock {\em Constr. Approx.}, 45(2):273--299, 2017.

\bibitem{Solomyak2024}
B.~Solomyak.
\newblock On nonlinear iterated function systems with overlaps.
\newblock {\em J. Fractal Geom.}, pages 01--11, 2024.

\bibitem{Verma2023}
M.~Verma, A.~Priyadarshi, and S.~Verma.
\newblock Analytical and dimensional properties of fractal interpolation
  functions on the \mbox{S}ierpi{\'n}ski gasket.
\newblock {\em Frac. Calc. Appl. Anal.}, 26(3):1294--1325, 2023.

\end{thebibliography}
	
\end{document}